\documentclass[a4paper,11pt]{article}
\pdfoutput=1

\usepackage[sec]{modstyle}

\title{Some remarks on higher Morita categories}
\author{Rune Haugseng}

\date{\today}

\begin{document}

\maketitle
\begin{abstract}
  We take another look at the construction of double
  \icats{} of algebras and bimodules and prove a few
  supplemental results about these, including a simpler proof of the
  Segal condition and a comparison between our construction
  and that of Lurie. We then take a more streamlined, inductive
  approach to the higher Morita categories of $E_{n}$-algebras
  and show that these deloop correctly.
\end{abstract}

\tableofcontents

\section{Introduction}
If $A,B,C$ are associative algebras in a reasonable monoidal
$\infty$-category $\uV$ and $M$ and $N$ are, respectively, an $A$-$B$-
and a $B$-$C$-bimodule, then we can define the \emph{relative tensor
  product} $M \otimes_{B} N$ as the canonical $A$-$C$-bimodule
structure on the colimit of a certain\footnote{For historical reasons, this is often known as the ``bar
  construction''.} simplicial object of the
form
\[ [n] \mapsto M \otimes B^{\otimes n} \otimes N,\] with inner face
maps given by multiplication in $B$, outer face maps by the action of
$B$ on $M$ and $N$, and degeneracies by inserting the unit in $B$. The
full functoriality of the relative tensor product construction can be
encoded as a double \icat{}, which we call the \emph{Morita double
  \icat{}} of $\uV$ due to its connection to Morita theory. This has
\begin{itemize}
\item associative algebras as objects,
\item homomorphisms of associative algebras as vertical morphisms,
\item bimodules as horizontal morphisms, with relative tensor products as their composition,
\item homomorphisms of bimodules as ``squares''.
\end{itemize}
Such a
double \icat{} was first constructed by Lurie in \cite{HA}*{\S 4.4},
and an alternative construction was given by the author in
\cite{nmorita}. There are also ``higher-dimensional'' versions of this
construction: if $\uV$ is a reasonable $E_{n}$-monoidal \icat{}, then
there is an $(n+1)$-fold \icat{} of $E_{n}$-algebras and iterated
bimodules in $\uV$. This was also discussed in \cite{nmorita}, and a
more geometric version using factorization algebras on certain
stratifications of $\mathbb{R}^{n}$ has been constructed by
Scheimbauer~\cite{ScheimbauerThesis}. These ``higher Morita
categories'' are interesting as targets for extended TQFTs, and higher
dualizability therein has been studied in
\cite{DouglasSchommerPriesSnyderDualTensorCat,GwilliamScheimbauer,BrochierJordanSnyder}
(where the latter use a variant of the construction due to
Johnson-Freyd and Scheimbauer~\cite{JohnsonFreydScheimbauerLax} to
obtain 4-categories of braided monoidal categories).

The aim of this note is to revisit the construction of higher
categories of $E_{n}$-algebras and iterated bimodules from
\cite{nmorita}, and prove a few results to supplement that paper.
The new results of the present paper are as follows:
\begin{itemize}
\item We give a new proof of the Segal condition for bimodules in
  \cref{sec:segcond}, much simpler than that in \cite{nmorita} (and
  perhaps also a bit easier than the one in \cite{HA}).
\item We show that the double \icat{} of algebras and bimodules
  constructed in \cite{nmorita} is equivalent to that of Lurie
  \cite{HA}*{\S 4.4} in \cref{sec:compare}.
\item In \cref{sec:iterate} we make fully explicit the inductive
  nature of the construction of $(n+1)$-fold \icats{} of $E_{n}$-algebras
  from \cite{nmorita}. Combined with a delooping result we prove in
  \cref{sec:deloop}, this allows us to show that the two constructions
  of an $E_{k}$-monoidal structure on the $n$-fold \icat{} of
  $E_{n}$-algebras in an $E_{n+k}$-monoidal \icat{} considered in
  \cite{nmorita} (by naturality and by delooping the $(n+k)$-fold
  \icat{} of $E_{n+k}$-algebras) are equivalent.
\end{itemize}

\section{Algebras, bimodules, and relative tensor products}
\label{sec:algebras}

In this section we briefly review the basic setup for the construction
of double \icats{} of algebras and bimodules; we refer the reader to
\cite{nmorita}*{\S 2} for motivation for these definitions. Here we
will freely use the algebraic framework of (generalized)
non-symmetric \iopds{}; see \cite{enr}*{\S 2} for motivation and
\cite{enr}*{\S 3} for a detailed discussion of these objects\footnote{But note that here we will denote the \icat{} of algebras for a \gnsiopd{} $\uO$ in a monoidal \icat{} $\uV$ as just $\Alg_{\uO}(\uV)$.}.

\begin{notation}
  For $[n] \in \simp$, we write
  \[ \Dopn := (\simp_{/[n]})^{\op} \cong (\Dop)_{[n]/}\]
  for the opposite of its overcategory in $\simp$. This is a \gnsiopd{} --- indeed a double \icat{} --- for any $[n]$ \cite{nmorita}*{Lemma 4.10} (or see \cref{pbarsimpdouble} below). It is sometimes convenient to describe an object of $\Dopn$ as a list $(i_{0},\ldots,i_{k})$ where $0 \leq i_{0} \leq \cdots \leq i_{k} \leq n$, with this corresponding to the morphism $[k] \to [n]$ that sends $j \in [k]$ to $i_{j}$.  
\end{notation}

Recall that in a monoidal \icat{} $\uV^{\otimes} \to \Dop$, we have that:
\begin{itemize}
\item Algebras for $\Dop = \Dop_{/[0]}$ are associative algebras; we write
  \[ \Alg(\uV) := \Alg_{\Dop}(\uV).\]
\item Algebras for $\Dop_{/[1]}$ describe two associative algebras (via the two copies of $\Dop$ in $\Dop_{/[1]}$) and a bimodule between them; we write
  \[ \Bimod(\uV) := \Alg_{\Dop_{/[1]}}(\uV). \]
\end{itemize}
In general, an algebra $A$ for $\Dopn$ describes:
\begin{itemize}
\item $n+1$ associative algebras $A(i,i)$ (via restriction along the functor $\Dop \to \Dopn$ given by composition with $[0] \cong \{i\} \hookrightarrow [n]$);
\item for all $i < j$, a bimodule $A(i,j)$ for the algebras $A(i,i)$ and $A(j,j)$;
\item for all $i < j < k$, an $A(j,j)$-bilinear map of $A(i,i)$-$A(k,k)$-bimodules $A(i,j) \otimes A(j,k) \to A(i,k)$;
\item compatibilities between these bilinear maps under tensoring and composition, for $n > 2$.
\end{itemize}
We want to single out those $\Dopn$-algebras where these bilinear maps exhibit $A(i,k)$ as a relative tensor product $A(i,j) \otimes_{A(j,j)} A(j,k)$:
\begin{defn}
  Let $\uV$ be a monoidal \icat{} compatible with simplicial
  colimits. We say a $\Dopn$-algebra $M \colon \Dopn \to \uV$ is
  \emph{composite} if for every $0 \leq i < j \leq n$, the canonical
  map
  \[ M(i,i+1) \otimes_{M(i+1,i+1)} \cdots \otimes_{M(j-1,j-1)} M(j-1,j) \to
    M(i,j)\]
  is an equivalence. We write $\Algc_{\Dopn}(\uV)$ for the full subcategory of $\Alg_{\Dopn}(\uV)$ spanned by the composite $\Dopn$-algebras.
\end{defn}
We can characterize the composite algebras as precisely those that are extended from a certain subobject of $\Dopn$:
\begin{defn}\label{defn:cellular}
  We say a morphism $\phi \colon [m] \to [n]$ in $\simp$ is
  \emph{cellular} if $\phi(i+1) \leq \phi(i)+1$ for all
  $i = 0,\ldots,m-1$. We then define $\bbL_{/[n]}$ to be the full
  subcategory of $\Dop_{/[n]}$ spanned by the cellular morphisms, and
  write $\Lopn := (\bbL_{/[n]})^{\op}$; the forgetful functor to
  $\Dop$ makes $\Lopn$ a \gnsiopd{} \cite{nmorita}*{Lemma 4.14} (or
  see \cref{lem:wLopgnsiopd}). Let $i_{n} \colon \Lopn \to \Dopn$
  denote the inclusion of opposite categories; this is a morphism of \gnsiopds{}.
\end{defn}

\begin{observation}
  A morphism $\phi \colon [m] \to [n]$ is cellular \IFF{} it is a composite of a surjective map followed by an inert map.
\end{observation}

\begin{propn}[\cite{nmorita}*{Corollary 4.20}]\label{propn:compalgs}
  Suppose $\uV$ is a monoidal \icat{} compatible with simplicial
  colimits. Then a $\Dopn$-algebra in $\uV$ is composite \IFF{} it is the operadic left Kan extension of its restriction to $\Lopn$. In other words, the restriction
  $i_{n}^{*} \colon \Alg_{\Dopn}(\uV) \to \Alg_{\Lopn}(\uV)$ has a
  fully faithful left adjoint $i_{n,!}$, given by operadic left Kan extension
  along $i_{n}$, with image $\Algc_{\Dopn}(\uV)$. \qed
\end{propn}

\section{The Segal condition}\label{sec:segcond}
In this section we will give a simple proof of the Segal condition
for composite bimodules. By this we mean the equivalences
\begin{equation}
  \label{eq:Segcond}
  \Alg_{\Lopn}(\uc{V}) \isoto \Bimod(\uc{V}) \times_{\Alg(\uc{V})} \cdots \times_{\Alg(\uc{V})} \Bimod(\uc{V})
\end{equation}
for every $n$ and any monoidal \icat{} $\uV$. This was first proved by Lurie as \cite{HA}*{Proposition 4.4.1.11} and then given a rather more complicated proof in \cite{nmorita}*{\S 4.3}.

We will prove \cref{eq:Segcond} first in the case where $\uV$ is compatible with small colimits. In this case we have access to free algebras for the \gnsiopds{} $\Lopn$, in the following sense:
\begin{propn}\label{propn:freealg}
  Let $\uO$ be a \gnsiopd{} such that $\uO_{[0]}$ is an \igpd{}. If $\uV$ is a monoidal \icat{} compatible with colimits indexed by \igpds{}, then the restriction
  \[ U_{\uO} \colon \Alg_{\uO}(\uV) \to \Fun(\uO_{[1]}, \uV)\]
  has a left adjoint $F_{\uO}$ such that for $\Phi \colon \uO_{[1]} \to \uV$ and $X \in \uO_{[1]}$ we have
  \[ T_{\uO}\Phi(X) \simeq \colimP_{\alpha \colon Y \to X}
      \Phi(Y_{01}) \otimes \cdots \otimes \Phi(Y_{(n-1)n})
    \]
   where $T_{\uO} := U_{\uO}F_{\uO}$ is the associated monad and the colimit is over the \igpd{} $\Act_{\uO}(X)$ of active morphisms $Y \to X$ in $\uO$, and for $Y \in \uO$ over $[n]$ in $\Dop$ we denote by $Y \to Y_{(i-1)i}$ a cocartesian morphism over the inert map $[1]\cong \{i-1,i\} \hookrightarrow [n]$ in $\simp$. Moreover, this adjunction is monadic.
\end{propn}
\begin{proof}
  For the existence of the left adjoint and the formula see the
  discussion in \cite{enr}*{\S A.4}; the monadicity follows from
  \cite{enr}*{Corollary A.5.6}. Alternatively, specialize
  \cite{patterns2}*{Corollary 8.14} using \cite{patterns2}*{Example
    9.6}.
\end{proof}

Since $(\Lopn)_{[1]}$ is isomorphic to the \emph{set}
\[T_{n} := \{(i,j) : 0 \leq i \leq j \leq i+1 \leq n\},\]
we get the following special case:
\begin{cor}\label{cor:freealgLopn}
If $\uV$ is a monoidal \icat{} compatible with colimits indexed by \igpds{}, then the restriction
  \[ U_{\Lopn} \colon \Alg_{\Lopn}(\uV) \to \Fun(T_{n}, \uV)\]
  has a left adjoint $F_{\Lopn}$ such that for $\Phi \colon T_{n} \to \uV$ we have
  \[ (T_{\Lopn}\Phi)(i,i) \simeq \coprod_{n = 0}^{\infty} \Phi(i,i)^{\otimes n},\]
  \[ (T_{\Lopn}\Phi)(i,i+1) \simeq \coprod_{n,m = 0}^{\infty} \Phi(i,i)^{\otimes n} \otimes \Phi(i,i+1) \otimes \Phi(i+1,i+1)^{\otimes m}.\]
  Moreover, the adjunction is monadic.
\end{cor}
\begin{proof}
  Apply \cref{propn:freealg}, observing that
  $\Act_{\Lopn}((i,j))$ is the set of commutative triangles
  \[
    \begin{tikzcd}
      {[1]} \arrow{rr}{\alpha} \arrow{dr}[swap]{(i,j)} & & {[m]} \arrow{dl}{\phi} \\
       & {[n],}
    \end{tikzcd}
  \]
  where $\alpha$ is the unique active map to $[m]$ (and $\phi$ is cellular). This forces $\phi$ to take only the values $i$ and $j$ (since $j$ is by assumption either $i$ or $i+1$), giving the desired description of the colimit over this set.
\end{proof}

\begin{observation}\label{obs:LopninertBCeq}
  With $\uV$ as above, we have for any inert morphism $\iota \colon [n] \to [m]$ a commutative square\footnote{Here we somewhat abusively use $\iota$ also for the induced map of \gnsiopds{} $\Lopn \to \Lop_{/[m]}$ and its restriction to a function $T_{n} \to T_{m}$.}
  \[
    \begin{tikzcd}
      \Alg_{\Lop_{/[m]}}(\uV) \arrow{r}{\iota^{*}}  \arrow{d}{U_{\Lop_{/[m]}}} & \Alg_{\Lopn}(\uV) \arrow{d}{U_{\Lopn}} \\
      \Fun(T_{m}, \uV) \arrow{r}{\iota^{*}} & \Fun(T_{n}, \uV).
    \end{tikzcd}
  \]
  The formula in 
  \cref{cor:freealgLopn} shows that the Beck--Chevalley transformation
  \[ F_{\Lopn} \iota^{*} \to \iota^{*} F_{\Lop_{/[m]}}\]
  is an equivalence. In other words, we have a commutative  mate square
  \[
    \begin{tikzcd}
      \Alg_{\Lop_{/[m]}}(\uV) \arrow{r}{\iota^{*}}   & \Alg_{\Lopn}(\uV)  \\
      \Fun(T_{m}, \uV) \arrow{r}{\iota^{*}}  \arrow{u}{F_{\Lop_{/[m]}}} & \Fun(T_{n}, \uV) \arrow{u}{F_{\Lopn}}.
    \end{tikzcd}
  \]
\end{observation}

\begin{notation}
  Let use write
  \[ \Alg^{\Seg}_{\Lopn}(\uc{V}) := \Bimod(\uc{V}) \times_{\Alg(\uc{V})} \cdots \times_{\Alg(\uc{V})} \Bimod(\uc{V}) \]
  for the $n$ times iterated fibre product, and $U_{\Lopn}^{\Seg} \colon \Alg^{\Seg}_{\Lopn}(\uV) \to
  \Fun(T_{n}, \uV)$ for the fibre product
  $U_{\Dop_{/[1]}} \times_{U_{\Dop}} \cdots \times_{U_{\Dop}} U_{\Dop_{/[1]}}$
  of the restrictons
  \[ U_{\Dop_{/[1]}} \colon \Bimod(\uV) \to \Fun(T_{1}, \uV), \quad U_{\Dop} \colon \Alg(\uV) \to \Fun(T_{0}, \uV),\]
  using the decomposition $T_{n} \cong T_{1}\amalg_{T_{0}} \cdots \amalg_{T_{0}} T_{1}$. We also write
  \[ j^{*}_{n,\Seg} \colon \Alg_{\Lopn}(\uV) \to \Alg^{\Seg}_{\Lopn}(\uV) \]
  for the functor induced by the inclusions $\Dop_{/[i]} \hookrightarrow \Lopn$, $i = 0,1$.
\end{notation}

The next results will allow us to identify the left adjoint of $U_{\Lopn}^{\Seg}$ and show that this is a monadic adjunction.

\begin{propn}\label{propn:ladjcommatelimadj}
  Suppose $\uc{A}, \uc{B} \colon \uc{I} \to \CatI$ are functors and $\rho \colon \uc{A} \to \uc{B}$ is a natural transformation such that the component $\rho(x) \colon \uc{A}(x) \to \uc{B}(x)$ has a left adjoint $\lambda(x)$ for every $x$. If the mate squares
  \[
    \begin{tikzcd}
      \uc{B}(x) \arrow{r}{\lambda(x)} \arrow{d}[swap]{\uc{B}(f)} &  \uc{A}(x) \arrow{d}{\uc{A}(f)} \\
      \uc{B}(x') \arrow{r}[swap]{\lambda(x')} \arrow[Rightarrow]{ur} & \uc{A}(x')
    \end{tikzcd}
  \]
  commute for all maps $f \colon x \to x'$ in $\uc{I}$, \ie{} if the
  Beck--Chevalley transformations
  $\lambda(x')\uc{B}(f) \to \uc{A}(f) \lambda(x)$ are all invertible,
  then the left adjoints form a natural transformation
  $\lambda \colon \uc{B} \to \uc{A}$, such that on limits the functor
  $\lim_{x \in \uc{I}} \rho(x) \colon \lim_{x \in \uc{I}} \uc{A}(x)
  \to \lim_{x \in \uc{I}} \uc{B}(x)$ is a right adjoint with left
  adjoint $\lim_{x \in \uc{I}} \lambda(x)$.
\end{propn}
\begin{proof}
  By \cite{HA}*{Proposition 7.3.2.11} these conditions imply that
  $\rho$ has an adjoint when viewed as a morphism of cocartesian
  fibrations. Since the straightening equivalence is an equivalence of
  \itcats{}, it follows that $\rho$ has a left adjoint in the \itcat{}
  $\FUN(\uc{I}, \CATI)$ of functors. In particular we have a unit and
  counit transformation in $\FUN(\uc{I}, \CATI)$ that in the limit
  give the unit and counit of the desired adjunction on limits.
\end{proof}

\begin{cor}\label{cor:ladjcommonradj}
  Suppose $\rho \colon \uc{A} \to \uc{B}$ and $\lambda \colon \uc{B} \to \uc{A}$ are as in \cref{propn:ladjcommatelimadj}. If furthermore $\rho(x)$ is a monadic right adjoint for every $x$, then $\lim_{x \in \uc{I}} \rho(x)$ is again a monadic right adjoint, and the associated monad has the limit $\lim_{x \in \uc{I}} \rho(x)\lambda(x)$ as its underlying endofunctor.
\end{cor}
\begin{proof}
  Let us write $\uc{A}' := \lim_{x \in \uc{I}} \uc{A}(x)$, etc. By the
  monadicity theorem we must check that $\rho'$ is conservative and
  that $\uc{A}'$ has $\rho'$-split simplicial colimits and these are
  preserved by $\rho'$. For the first point, suppose a morphism
  $\alpha \colon a \to a'$ in $\uc{A}'$ maps to an equivalence in
  $\uc{B}'$. Then its images in $\uc{B}(x)$ are all equivalences, so
  by the componentwise conservativity of $\rho$ the image of $\alpha$
  in each $\uc{A}(x)$ is an equivalence. This implies that $\alpha$ is
  an equivalence, since the projections from a limit in $\CatI$ are
  jointly conservative \cite{HA}*{Proposition 5.2.2.36}. Next, suppose
  we have a $\rho'$-split simplicial diagram
  $\phi \colon \Dop \to \uc{A}'$, so that there exists an extension
  $\psi$ of $\rho\phi$ to a split simplicial diagram in
  $\uc{B}'$. Then the projection $\psi_{x}$ of $\psi$ in $\uc{B}(x)$
  is a split simplicial diagram that extends the image under $\rho(x)$
  of the projection $\phi_{x}$ of $\phi$ to $\uc{A}(x)$. In other
  words, $\phi_{x}$ is a $\rho(x)$-split simplicial diagram in
  $\uc{A}(x)$; by assumption it therefore has a colimit and this is
  preserved by $\rho(x)$. By \cite{LiBlandSpan}*{Theorem 2.1} this means that $\phi$
  has a colimit in $\uc{A}'$ and this is preserved by $\rho'$, as
  required. The description of the associated monad simply follows
  from functoriality of limits, since the left adjoint is
  $\lim_{x \in \uc{I}} \lambda(x)$.
\end{proof}

\begin{propn}\label{Segcondcolim}
  Let $\uV$ be a monoidal \icat{} compatible with colimits indexed by \igpds{}. Then the restriction
  \[ \Alg_{\Lopn}(\uV) \to \Alg^{\Seg}_{\Lopn}(\uV) \]
  is an equivalence.
\end{propn}

\begin{proof}
  We have a commutative triangle
  \[
    \begin{tikzcd}
      \Alg_{\Lopn}(\uc{V}) \arrow{rr}{j_{n,\Seg}^{*}} \arrow{dr}[swap]{U_{\Lopn}} & &  \Alg^{\Seg}_{\Lopn}(\uc{V}) \arrow{dl}{U_{\Lopn}^{\Seg}} \\
       & \Fun(T_{n}, \uV).
    \end{tikzcd}
  \]
  Here both the downward functors are monadic right adjoints: For $U_{\Lopn}$ this holds by \cref{cor:freealgLopn} and for $U^{\Seg}_{\Lopn}$ by the same combined with \cref{cor:ladjcommonradj}, since we saw in \cref{obs:LopninertBCeq} that the associated Beck--Chevalley transformations are invertible. To show that the top horizontal functor is an equivalence it therefore suffices by
  \cite{HA}*{Corollary 4.7.3.16} to show that the induced map of monads is an equivalence. If we write $F_{\Lopn}^{\Seg}$ for the left adjoint to $U_{\Lopn}^{\Seg}$ and $T_{\Lopn}^{\Seg} := U_{\Lopn}^{\Seg}F_{\Lopn}^{\Seg}$ for the corresponding monad, then this amounts to proving that the natural transformation
  \[ T_{\Lopn} \to T_{\Lopn}^{\Seg} \simeq T_{\Dop_{/[1]}}
    \times_{T_{\Dop}} \cdots \times_{T_{\Dop}} T_{\Dop_{/[1]}} \]
  between endofunctors of $\Fun(T_{n}, \uc{V})$ is an
  equivalence. This follows from the formula for $T_{\Lopn}$ in
  \cref{cor:freealgLopn} and the compatibility we noted in
  \cref{obs:LopninertBCeq}.
\end{proof}

\begin{cor}\label{cor:Segcond}
  For any monoidal \icat{} $\uV$, the restriction
  \[ j_{n,\Seg}^{*} \colon \Alg_{\Lopn}(\uV) \to \Alg_{\Lopn}^{\Seg}(\uV)\]
  is an equivalence.
\end{cor}
\begin{proof}
  We may assume without loss of generality that $\uV$ is small. (If not, we simply apply the same argument in a larger universe.) Then the Yoneda embedding $y \colon \uV \hookrightarrow \PSh(\uV)$ exhibits $\uV$ as a full monoidal subcategory of the Day convolution monoidal structure on presheaves \cite{patterns2}*{\S 6}. We have a commutative square

  \[
    \begin{tikzcd}
      \Alg_{\Lopn}(\uV) \arrow{r}{j_{n,\Seg}^{*}} \arrow[hookrightarrow]{d}{y_{*}} & \Alg^{\Seg}_{\Lopn}(\uV) \arrow[hookrightarrow]{d}{y_{*}} \\
      \Alg_{\Lopn}(\PSh(\uV)) \arrow{r}{j_{n,\Seg}^{*}} & \Alg^{\Seg}_{\Lopn}(\PSh(\uV))
    \end{tikzcd}
  \]
where the bottom horizontal functor is an equivalence by \cref{Segcondcolim} since $\PSh(\uV)$ is compatible with colimits. Moreover, for any \gnsiopd{} $\uO$, the functor
\[ y_{*} \colon \Alg_{\uO}(\uV) \to \Alg_{\uO}(\PSh(\uV))\] given by
composition with $y$ is fully faithful with image those $\uO$-algebras
whose restrictions to $\uO_{[1]}$ factor through $\uV$. It follows
that our equivalence for presheaves identifies the full subcategories
$\Alg_{\Lopn}(\uV)$ and $\Alg^{\Seg}_{\Lopn}(\uV)$ with each other, so
that $j_{n,\Seg}^{*}$ is indeed an equivalence between them.
\end{proof}

\begin{remark}
  The equivalence of \cref{cor:Segcond} can be extended to the case
  where $\uV$ is a non-symmetric \iopd{}, by embedding it in its
  monoidal envelope and checking that the equivalence restricts
  appropriately. The version proved in \cite{nmorita} is actually a
  bit more general than this, since it shows that $\Lopn$ is an
  iterated pushout
  $\Dop_{/[1]} \amalg_{\Dop} \cdots \amalg_{\Dop} \Dop_{/[1]}$ in the
  \icat{} of \gnsiopds{}. This implies that the Segal condition
  actually holds for $\uV$ a \emph{generalized} \nsiopd{}, which we do
  not know how to show using the method of proof used here.
\end{remark}

\section{The double $\infty$-category of algebras and bimodules}\label{sec:double}
In this section we review the definition of the double \icat{} of
algebras and bimodules as a cocartesian fibration over $\Dop$. Here we
set this up in a way that makes it clear how it maps to Lurie's
version, as we will see in the next section.

\begin{notation}
  For any \icat{} $\uC$, we write $\Ar(\uC) := \Fun([1], \uC)$ for its \icat{} of arrows, and $\ev_{i} \colon \Ar(\uC) \to \uC$ for the functors given by evaluation at $i \in [1]$ ($i = 0,1$).
\end{notation}

\begin{observation}
  For any \icat{} $\uC$, the functor $\ev_{0} \colon \Ar(\uC) \to \uC$
  is a cartesian fibration, while $\ev_{1}$ is a cocartesian
  fibration. Moreover, a morphism in $\Ar(\uC)$, that is a commutative
  square
  \[
    \begin{tikzcd}
      x \arrow{d}{f} \arrow{r}{u} & x' \arrow{d}{f'} \\
      y \arrow{r}{v} & y'
    \end{tikzcd}
  \]
  is $\ev_{1}$-cocartesian \IFF{} $u$ is an equivalence and $\ev_{0}$-cartesian \IFF{} $v$ is an equivalence. Thus the commutative triangles
  \[
    \begin{tikzcd}
      \Ar(\uC) \arrow{rr}{(\ev_{0},\ev_{1}) } \arrow{dr}{\ev_{0}} & & \uC \times \uC \arrow{dl}{\pr_{1}} \\
       & \uC,
     \end{tikzcd} \qquad
    \begin{tikzcd}
      \Ar(\uC) \arrow{rr}{(\ev_{0},\ev_{1}) } \arrow{dr}{\ev_{1}} & & \uC \times \uC \arrow{dl}{\pr_{2}} \\
       & \uC,
    \end{tikzcd}     
  \]
  are morphisms of cartesian and cocartesian fibrations, respectively.
\end{observation}

\begin{propn}\label{propn:Arsimpopdouble}
  The cocartesian fibration $\ev_{0}^{\op} \colon \Ar(\simp)^{\op} \to \Dop$ is a double \icat{}.
\end{propn}
\begin{proof}
  To see that $\ev_{0}^{\op}$ is a double \icat{}, we must show that the functor
  \[ \simp_{[n]/} \to \simp_{[1]/} \times_{\simp_{[0]/}} \cdots
    \times_{\simp_{[0]/}} \simp_{[1]/}, \] given by composition with
  the inert maps $[0],[1] \hookrightarrow [n]$, is an
  equivalence. This follows from the elementary observation that $[n]$
  is the pushout \[[1] \amalg_{[0]} \cdots \amalg_{[0]} [1]\] in
  $\simp$.
\end{proof}

\begin{cor}\label{pbarsimpdouble}
  For any functor of \icats{} $\uc{K}\to \Dop$, the pullback $\uc{K} \times_{\Dop} \Ar(\simp)^{\op}$ along $\ev_{1}^{\op}$ is a double \icat{} via the functor
  \[ \uc{K} \times_{\Dop} \Ar(\simp)^{\op} \to \Ar(\simp)^{\op} \xto{\ev_{0}^{\op}} \Dop.\]
\end{cor}
\begin{proof}
  The projection $\uc{C} \times \Dop \to \Dop$ is a double \icat{} for any \icat{} $\uC$, so we have a pullback square
  \begin{equation}
    \label{eq:ardoppbsq}
    \begin{tikzcd}
      \uc{K} \times_{\Dop} \Ar(\simp)^{\op} \arrow{r}{\pr_{2}} \arrow{d}[swap]{(\pr_{1}, \ev_{0}^{\op}\pr_{2})} & \Ar(\simp)^{\op}  \arrow{d}{(\ev_{1}^{\op}, \ev_{0}^{\op})}\\
      \uc{K} \times \Dop \arrow{r}{\phi \times \id} & \Dop \times \Dop
    \end{tikzcd}
  \end{equation}
  of double \icats{}.
\end{proof}

\begin{propn}\label{propn:omor'eq}
  Let $\oMor'(\uV) \to \Dop$ be the cocartesian fibration for the functor $\Fun_{/\Dop}(\Dop_{/\blank}, \uV)$. Then there is a natural equivalence
  \[ \Fun_{/\Dop}(\uc{K}, \oMor'(\uV)) \simeq \Fun_{\Dop}(\uc{K} \times_{\Dop} \Ar(\simp)^{\op}, \uV).\]
  In other words, $\oMor'(\uV)$ is $\ev^{\op}_{1,*}\ev^{\op,*}_{1}\uV$.
\end{propn}
\begin{proof}
  The functor $\Dop_{/\blank} \colon \simp \to \Cat$ is classified by
  the cartesian fibration \[\ev_{1}^{\op}\colon \Ar(\simp)^{\op} \to \Dop.\] This is therefore a special case of
  \cite{freepres}*{Proposition 7.3}.
\end{proof}

\begin{defn}
  Let $\uV$ be a monoidal \icat{}. We define $\oMor(\uV)$ to be the full subcategory of $\oMor'(\uV)$ spanned by the $\Dopn$-algebras for all $n$, so that the restricted projection $\oMor(\uV) \to \Dop$ is the cartesian fibration for the functor $\Alg_{\Dop_{/\blank}}(\uV)$.
\end{defn}

\begin{propn}\label{propn:oMorftrisalg}
  For any \icat{} $\uc{K}$ over $\Dop$, the equivalence of \cref{propn:omor'eq} restricts to a natural equivalence
  \[ \Fun_{/\Dop}(\uc{K}, \oMor(\uV)) \simeq \Alg_{\uc{K} \times_{\Dop} \Ar(\simp)^{\op}}(\uV).\]
\end{propn}
\begin{proof}
  By definition, a functor $\uc{K} \to \oMor'(\uV)$ factors through $\oMor(\uV)$ precisely when the image of each object of $\uc{K}$ lies in $\oMor(\uV)$. Translating this through the equivalence of \cref{propn:omor'eq}, it says that functors $\uc{K} \to \oMor(\uV)$ correspond to functors $\uc{K} \times_{\Dop} \Ar(\simp)^{\op} \to \uc{V}$ such that for every object $x \in \uc{K}$ over $[n]$ in $\Dop$, the restriction to $\Dopn \simeq \{x\} \times_{\Dop} \Ar(\simp)^{\op} \to \uV$ preserves inert morphisms. We therefore want to show that this condition is equivalent to preserving the inert morphisms in $\uc{K} \times_{\Dop} \Ar(\simp)^{\op}$. But from the pullback square \cref{eq:ardoppbsq} we see that a morphism in $\uc{K} \times_{\Dop} \Ar(\simp)^{\op}$ is cocartesian precisely when its image in $\uc{K}$ is an equivalence, so the inert morphisms are indeed precisely those that are the images of inert morphisms in $\Dopn$ for some object of $\uc{K}$.
\end{proof}

\begin{defn}
  Let $\Mor(\uV)$ be the full subcategory of $\oMor(\uV)$ spanned by the composite $\Dopn$-algebras for all $n$.
\end{defn}

We want to show that $\Mor(\uV) \to \Dop$ is a cocartesian fibration, with the cocartesian morphisms inherited from $\oMor(\uV)$. In other words, we want to prove that for any morphism $\phi \colon [n] \to [m]$ in $\simp$, the functor $\phi^{*} \colon \Alg_{\Dop_{/[m]}}(\uV) \to \Alg_{\Dopn}(\uV)$ preserves composite algebras. To see this, we first introduce a family of auxiliary objects:
\begin{defn}
  Given a morphism $\phi \colon [n] \to [m]$ in $\simp$, we say a morphism $\alpha \colon [k] \to [m]$ is \emph{$\phi$-cellular} if the following conditions hold:
  \begin{itemize}
  \item if $\alpha(i) < \phi(0)$ then $\alpha(i+1) \leq \alpha(i) + 1$,
  \item if $\phi(j) \leq \alpha(i) < \phi(j+1)$ then $\alpha(i+1) \leq \phi(j+1)$.
  \item if $\alpha(i)> \phi(n)$ then  $\alpha(i+1) \leq \alpha(i) + 1$.
  \end{itemize}
  We write $\bbL_{/[m]}[\phi]$ for the full subcategory of
  $\simp_{/[m]}$ spanned by the $\phi$-cellular morphisms, and denote
  the inclusion by
  $i_{\phi} \colon \bbL_{/[m]}[\phi] \hookrightarrow
  \simp_{/[m]}$. Then $\bbL_{/[m]}[\phi]$ contains $\bbL_{/[m]}$ ---
  we write
  $j_{\phi} \colon \bbL_{/[m]} \hookrightarrow \bbL_{/[m]}[\phi]$ for
  the inclusion --- and the functor
  $\phi \colon \simp_{/[n]} \to \simp_{/[m]}$ restricts to a functor
  $\phi_{\bbL} \colon \bbL_{/[n]} \to \bbL_{/[m]}[\phi]$. In total, we
  have the commutative diagram
  \begin{equation}
    \label{eq:phicellsq}
    \begin{tikzcd}
      \bbL_{/[n]} \arrow{r}{\phi_{\bbL}} \arrow[hookrightarrow]{d}[swap]{i_{n}} & \bbL_{/[m]}[\phi] \arrow[hookrightarrow]{d}{i_{\phi}} & \bbL_{/[m]} \arrow[hookrightarrow]{l}[swap]{j_{\phi}} \arrow[hookrightarrow]{dl} \\
      \simp_{/[n]} \arrow{r}{\phi} & \simp_{/[m]}.
    \end{tikzcd}
  \end{equation}
\end{defn}

\begin{observation}\label{obs:phicellsurj}
  If $\phi \colon [n] \to [m]$ is surjective, then $\bbL_{/[m]}[\phi]$ is the same as $\bbL_{/[m]}$. More generally, if $\phi$ decomposes as $[n] \xto{\alpha} [k] \xto{\beta} [m]$ where $\alpha$ is surjective, then $\bbL_{/[m]}[\phi]$ is the same as $\bbL_{/[m]}[\beta]$.
\end{observation}

\begin{lemma}
  For any map $\phi \colon [n] \to [m]$, the forgetful functor to $\Dop$ makes $\Lop_{/[m]}[\phi]$ a \gnsiopd{}, and the inclusions
  \[ \Lop_{/[m]} \hookrightarrow \Lop_{/[m]}[\phi] \hookrightarrow
    \Dop_{/[m]}\] are extendable morphisms of \gnsiopds{}.
\end{lemma}
\begin{proof}
  We can apply \cite{nmorita}*{Lemma 4.15} to conclude that
  $\Lop_{/[m]}[\phi]$ is a \gnsiopd{}, since a morphism $[k] \to [m]$
  is $\phi$-cellular \IFF{} all its inert restrictions to $[1]$ are
  $\phi$-cellular. Both inclusions restrict on the fibre over $[0]$ to
  the identity of the set $\Hom_{\simp}([0],[m])$, and so are
  extendable since this is true for any morphism of \gnsiopds{} that is an equivalence of \igpds{} over $[0]$ by \cite{patterns2}*{Example 9.6}.
\end{proof}

\begin{propn}\label{propn:phicelllke}
  Let $\uV$ be a monoidal \icat{} compatible with simplicial colimits. Then the commutative diagram \cref{eq:phicellsq} induces a commutative diagram
  \begin{equation}
    \label{eq:phicelllke}
    \begin{tikzcd}
      \Alg_{\Lop_{/[n]}}(\uV) \arrow{d}{i_{n,!}} & \Alg_{\Lop_{/[m]}[\phi]}(\uV) \arrow{l}{\phi_{\bbL}^{*}} \arrow{d}{j_{\phi,!}}  & \Alg_{\Lop_{/[m]}}(\uV) \arrow{l}{i_{\phi,!}} \arrow{dl}{i_{m,!}} \\
      \Alg_{\Dop_{/[n]}}(\uV)  & \Alg_{\Dop_{/[m]}}(\uV). \arrow{l}{\phi^{*}}
    \end{tikzcd}
  \end{equation}
\end{propn}

The following is the key technical point of the proof:
\begin{lemma}\label{lem:Lphicoinit}
  For any maps $\phi \colon [n] \to [m]$ and $\psi \colon [k] \to [n]$ in $\simp$, the functor
  \[ (\bbLn^{\act})_{\psi/} := \bbLn^{\act} \times_{\Dn^{\act}} (\Dn^{\act})_{\psi/} \to
    \bbL_{/[m]}[\phi]^{\act} \times_{\simp_{/[m]}^{\act}} (\simp_{/[m]}^{\act})_{\phi\psi/}        =:
    (\bbL_{/[m]}[\phi])^{\act}_{\phi\psi/}\]
  is coinitial.
\end{lemma}
\begin{proof}
  This is proved as \cite{nmorita}*{Proposition 4.37} when $\phi$ is either injective or surjective. The general case follows from this, since we may factor $\phi$ as $[n] \xto{\alpha} [l] \xto{\beta} [m]$ where $\alpha$ is surjective and $\beta$ is injective. The functor we want for $\phi$ is then the composite
  \[ (\bbLn^{\act})_{\psi/} \to
    (\bbL_{/[l]}[\alpha])^{\act}_{\alpha\psi/} \cong
    (\bbL_{/[l]}^{\act})_{\alpha\psi/} \to
    (\bbL_{/[l]}[\beta]^{\act})_{\phi\psi/} \cong
    (\bbL_{/[l]}[\phi]^{\act})_{\phi\psi/},\] using the identifications
  of \cref{obs:phicellsurj}.
\end{proof}

\begin{proof}[Proof of \cref{propn:phicelllke}]
  We must show that the Beck--Chevalley transformation
  \[ i_{n,!}\phi^{*}_{\Lambda} \to \phi^{*}j_{\phi,!}\]
  is an equivalence. Using the formula for operadic left Kan extensions, we see that at an object $\phi \colon [k]  \to [n]$ in $\Dopn$ this is given for any $\Lop_{/[m]}[\phi]$-algebra as a map on colimits induced by the opposite of the functor from \cref{lem:Lphicoinit}. Since this is cofinal (being the opposite of a coinitial functor), the induced map is indeed an equivalence.
\end{proof}

\begin{cor}\label{cor:compprescomposite}
  For any morphism $\phi \colon [n] \to [m]$ in $\simp$, the induced functor \[\phi^{*} \colon \Alg_{\Dop_{/[m]}}(\uV) \to \Alg_{\Dop_{/[n]}}(\uV) \] preserves composite algebras.
\end{cor}
\begin{proof}
  Suppose $X \in \Alg_{\Dop_{/[m]}}(\uV)$ is a composite algebra. Then by \cref{propn:compalgs} we can write $X$ as $i_{m,!}X'$ for a unique $\Lop_{/[n]}$-algebra $X'$. From \cref{propn:phicelllke} we then get a natural equivalence
  \[ \phi^{*}X \simeq \phi^{*}i_{m,!}X' \simeq i_{n,!}\phi^{*}_{\Lambda}i_{\phi,!}X',\]
  so that $\phi^{*}X$ is in the image of $i_{n,!}$ and so composite.
\end{proof}

\begin{cor}\label{cor:Morcocart}
  The restricted projection $\Mor(\uV) \to \Dop$ is a cocartesian fibration.
\end{cor}
\begin{proof}
  It follows from \cref{cor:compprescomposite} that if $X \to Y$ is a cocartesian morphism in $\oMor(\uV)$ where $X$ lies in $\Mor(\uV)$, then so does $Y$. The full subcategory $\Mor(\uV)$ therefore inherits the required cocartesian morphisms from $\oMor(\uV)$. 
\end{proof}

\begin{observation}\label{obs:cocartinmor}
  If we think of the objects of $\Mor(\uV)$ as $\Lopn$-algebras for
  all $n$, then the cocartesian transport functor over
  $\phi \colon [n] \to [m]$ in $\simp$ can be described as the
  composite
  \[ \Alg_{\Lop_{/[m]}}(\uV) \xto{j_{\phi,!}} \Alg_{\Lop[\phi]}(\uV) \xto{\phi_{\bbL}^{*}} \Alg_{\Lop_{/[n]}}(\uV).\]
\end{observation}

Combining \cref{cor:Morcocart} with \cref{cor:Segcond}, we have shown:
\begin{cor}
  The projection $\Mor(\uV) \to \Dop$ is a double \icat{}.
\end{cor}

\section{Comparison with Lurie's version}\label{sec:compare}
In this section we will show that the double \icat{} $\Mor(\uV)$ we constructed in the previous section is equivalent to that defined by Lurie in \cite{HA}*{\S 4.4}. We start by briefly reviewing Lurie's construction.

\begin{defn}
  Let $\Arc(\simp)$ denote the full subcategory of $\Ar(\simp)$ spanned by the cellular morphisms (in the sense of \cref{defn:cellular}), and let $\ev_{1}^{c} \colon \Arc(\simp) \to \simp$ denote the restriction of evaluation at $1$ in $[1]$ to this full subcategory. Note that the fibre of $\ev_{1}^{c}$ at $[n]$ is then $\bbL_{/[n]}$. We will also use the notation
  \[ \wL := \Arc(\simp). \]
\end{defn}

\begin{lemma}\label{lem:wLopgnsiopd}
  $\ev^{c,\op}_{0} \colon \wL^{\op} \to \Dop$ is a \gnsiopd{}, and for any functor $\uc{K} \to \Dop$, so is the pullback
  $\uc{K} \times_{\Dop} \wL^{\op}$ over $\ev^{c,\op}_{1}$ via the composite
\[ \uc{K} \times_{\Dop} \wL^{\op} \to \wL^{\op} \xto{\ev^{c,\op}_{0}} \Dop. \]
\end{lemma}
\begin{proof}
  We can apply \cite{nmorita}*{Lemma 4.15} to the inclusion of
  $\wL^{\op}$ in $\Ar(\simp)^{\op}$ (which is a double \icat{} by
  \cref{propn:Arsimpopdouble}) to conclude that the former is a
  \gnsiopd{}, since a morphism $[k] \to [m]$ is cellular \IFF{} all
  its inert restrictions to $[1]$ are cellular. As in the proof of \cref{pbarsimpdouble}, we then have for any functor $\uc{K} \to \Dop$ a pullback
  \[
    \begin{tikzcd}
      \uc{K} \times_{\Dop} \wLop \arrow{r}{\pr_{2}} \arrow{d}[swap]{(\pr_{1}, \ev_{0}^{\op}\pr_{2})} & \wLop  \arrow{d}{(\ev_{1}^{\op}, \ev_{0}^{\op})}\\
      \uc{K} \times \Dop \arrow{r}{\phi \times \id} & \Dop \times \Dop
    \end{tikzcd}
  \]
  of \gnsiopds{}.
\end{proof}

\begin{propn}[\cite{HA}*{Proposition 4.4.3.21}]
  $\ev^{c}_{1} \colon \Arc(\simp) \to \simp$ is an exponentiable fibration\footnote{In other words, pullback along this functor preserves colimits, and so has a right adjoint. Such functors are also known as \emph{flat} or \emph{Conduch\'e} fibrations.}. \qed
\end{propn}

\begin{cor}
  For any \icat{} $\uc{C}$ over $\Dop$ there exists an \icat{} \[\oMorL(\uc{C}) := \ev^{c}_{1,*}\ev^{c,*}_{1}\uc{C}\] over $\Dop$ such that
  \[ \Fun_{/\Dop}(\uK, \oMorL(\uc{C})) \simeq \Fun_{/\Dop}(\uK \times_{\Dop} \Arc(\simp)^{\op}, \uc{C})\]
  for any \icat{} $\uK$ over $\Dop$.
\end{cor}
\begin{proof}
  The right adjoint $\ev^{c}_{1,*}$ gives natural equivalences of \igpds{}
  \[ \Map_{/\Dop}(\uK, \oMorL(\uc{C})) \simeq \Map_{/\Dop}(\uK \times_{\Dop} \Arc(\simp)^{\op}, \uc{C}). \]
  To upgrade this to an equivalence of \icats{}, we simply observe that for any \icat{} $\uD$ we get
  \[ 
    \begin{split}
     \Map_{\CatI}(\uD, \Fun_{/\Dop}(\uK, \oMorL(\uc{C}))) & 
                                                            \simeq \Map_{/\Dop}(\uD \times \uK, \oMorL(\uc{C})) \\
                                                          & \simeq \Map_{/\Dop}((\uD \times \uK) \times_{\Dop} \Arc(\simp)^{\op}, \uc{C}) \\
       & \simeq \Map_{/\Dop}(\uD \times (\uK \times_{\Dop} \Arc(\simp)^{\op}), \uc{C}) \\
       & \simeq \Map_{\CatI}(\uD, \Fun_{/\Dop}(\uK \times_{\Dop} \Arc(\simp)^{\op}, \uc{C})),
    \end{split}
  \]
  and then apply the Yoneda lemma.
\end{proof}

\begin{observation}
  An object of $\oMorL(\uc{C})$ over $[n]$ is a commutative triangle
  \[
    \begin{tikzcd}
      \bbL_{/[n]}^{\op} \arrow{dr} \arrow{rr} & & \uc{C} \arrow{dl} \\
       & \Dop
    \end{tikzcd}
  \]
\end{observation}

\begin{defn}
  Suppose $\uV$ is a monoidal \icat{}. We define $\MorL(\uV)$ to be the full subcategory of $\oMorL(\uV)$ spanned by the $\bbL_{/[n]}^{\op}$-algebras for all $[n]$.
\end{defn}

\begin{lemma}
  For any functor $\uc{K} \to \Dop$, there is a natural equivalence
  \[ \Fun_{/\Dop}(\uc{K}, \MorL(\uV)) \simeq \Alg_{\uc{K}\times_{\Dop} \wLop}(\uV).\]
\end{lemma}
\begin{proof}
  Same as \cref{propn:oMorftrisalg}.
\end{proof}

\begin{notation}
  For a morphism $\phi \colon [n] \to [m]$ in $\simp$, we write
  $\wLop_{\phi}$ for the \gnsiopd{} $[1] \times_{\Dop} \wLop$ obtained
  by pullback along $[1] \xto{\phi} \Dop$. Taking the fibres over $0,1 \in [1]$ gives fully faithful inclusions
  \[ \Lop_{/[m]} \xto{u_{\phi}} \wLop_{\phi} \xfrom{v_{\phi}} \Lop_{/n}.\]
\end{notation}

\begin{lemma}
  For any morphism $\phi$, the inclusion
  \[ u_{\phi} \colon \Lop_{/[m]} \hookrightarrow \wLop_{\phi}\]
  is an extendable morphism of \gnsiopds{}.
\end{lemma}
\begin{proof}
  We must show that for an object $X$ of $\wLop_{\phi}$ over $[n]$, the canonical functor
  \[ (\Lop_{/[m]})^{\act}_{/X} \to \prod_{i=1}^{n}(\Lop_{/[m]})^{\act}_{/X_{(i-1)i}} \]
  is cofinal, where $X \to X_{(i-1)i}$ is cocartesian over $\{i,i+1\} \hookrightarrow [n]$.
  But for any morphism of \gnsiopds{} $f \colon \uc{O} \to \uc{P}$ and $P \in \uc{P}$ over $[n]$, we have an equivalence
  \[ \uc{O}^{\act}_{/P} \to \uc{O}^{\act}_{/P_{1}} \times_{\uc{O}^{\act}_{/P_{11}}} \cdots \times_{\uc{O}^{\act}_{/P_{n-1}}}
    \uc{O}^{\act}_{/P_{(n-1)n}} \]
  where $P \to P_{i}$ is cocartesian over $\{i\} \hookrightarrow [n]$ and $P \to P_{i(i+1)}$ over $\{i,i+1\} \hookrightarrow [n]$ \cite{patterns1}*{Proposition 9.15}. To show that $f$ is extendable it therefore suffices to prove that the \icat{} $\uc{O}^{\act}_{/P} \simeq \uc{O}^{\act}_{0/P}$ is contractible for every $P \in \uc{P}_{0}$. In our case $X$ is either a cellular map $\beta \colon [a] \to [m]$, in which case
  $(\Lop_{/[m]})^{\act}_{/X}$ consists of triangles
  \[
    \begin{tikzcd}
      {[a]} \arrow{rr}{\alpha} \arrow{dr}[swap]{\beta} & & {[b]} \arrow{dl}{\gamma} \\
       & {[m]},
    \end{tikzcd}
  \]
  or a cellular map $\alpha \colon [a] \to [n]$, in which case
    $(\Lop_{/[m]})^{\act}_{/X}$ consists of squares
  \[
    \begin{tikzcd}
      {[a]} \arrow{r}{\alpha} \arrow{d}{\beta} & {[b]} \arrow{d}{\gamma} \\
      {[n]} \arrow{r}{\phi} & {[m]},
    \end{tikzcd}
  \]
  where in both cases $\alpha$ is active and $\gamma$ is cellular. For
  $a = 0$, the condition that $\alpha$ is active forces $b = 0$, and
  then $\gamma$ is uniquely determined.
\end{proof}

\begin{thm}[Lurie]\label{thm:MorLur}
  Suppose $\uV$ is a monoidal \icat{} compatible with simplicial colimits. Then $\MorL(\uV) \to \Dop$ is a cocartesian fibration, and the cocartesian transport functor over $\phi \colon [n] \to [m]$ in $\simp$ is given by the composite
  \[ \Alg_{\Lop_{/[m]}}(\uV) \xto{u_{\phi,!}} \Alg_{\wLop_{\phi}}(\uV) \xto{v_{\phi}^{*}} \Alg_{\Lopn}(\uV) \]
  of operadic left Kan extension along $u_{\phi}$ and restriction along $v_{\phi}$. If $\phi$ is inert then this is equivalently just restriction along the induced map $\Lop_{/[n]} \to \Lop_{/[m]}$.
\end{thm}
\begin{proof}
  Under the given assumption this functor is a cocartesian fibration by \cite{HA}*{Lemma 4.4.3.9(1)} and the description of cocartesian morphisms follows in general from \cite{HA}*{Corollary 4.4.3.2} and in the inert case from \cite{HA}*{Lemma 4.4.3.9(2)}.
\end{proof}

Combining this with the Segal condition from \cref{cor:Segcond} or \cite{HA}*{Proposition 4.4.1.11}, we have:
\begin{cor}
  Suppose $\uV$ is a monoidal \icat{} compatible with simplicial colimits. Then $\MorL(\uV) \to \Dop$ is a double \icat{}. \qed
\end{cor}

\begin{construction}
  The inclusion $\Arc(\simp) \hookrightarrow \Ar(\simp)$ induces for any $\uc{C}$ over $\Dop$ a functor
  \begin{equation}
    \label{eq:morfunpre}
  \oMor'(\uc{C}) \to \oMorL(\uc{C})  
  \end{equation}
  over $\Dop$.
\end{construction}

\begin{propn}
  Suppose $\uc{V}$ is a monoidal \icat{} compatible with simplicial colimits. Then the functor \cref{eq:morfunpre} restricts to a functor
  \[ \Mor(\uc{V}) \to \MorL(\uc{V})\]
  over $\Dop$, which moreover preserves cocartesian morphisms. 
\end{propn}
\begin{proof}
  On fibres, the functor is given by restriction along
  $i_{n} \colon \Lop_{/[n]} \to \Dop_{/[n]}$. Since this preserves
  algebras, the functor restricts as required. To see that the
  restriction preserves cocartesian morphisms, we use the description
  of these in \cref{obs:cocartinmor} and \cref{thm:MorLur}. For
  $\phi \colon [n] \to [m]$ in $\simp$, there is a canonical functor
  $q_{\phi} \colon \wL_{\phi} \to \bbL_{/[m]}[\phi]$ (given by the identity on the image of $\bbL_{/[m]}$ and by composition with $\phi$ on the image of $\bbL_{/[n]}$ in $\wL_{\phi}$), which fits in a commutative diagram
  \[
    \begin{tikzcd}
      {} & \wL_{\phi}  \arrow{dd}{q_{\phi}}  \\
      \bbL_{/[m]} \arrow[hook]{ur}{u_{\phi}} \arrow{dr}[swap]{j_{\phi}} & & \bbL_{/[n]} \arrow[hook']{ul}[swap]{v_{\phi}} \arrow[hook']{dl}{\phi_{\Lambda}} \\
       & \bbL_{/[m]}[\phi].
    \end{tikzcd}
  \]
  It suffices to check that on algebras this induces a commutative triangle
  \[
    \begin{tikzcd}
      \Alg_{\Lop_{/[m]}}(\uV) \arrow{r}{j_{\phi,!}} \arrow{dr}[swap]{u_{\phi,!}} & \Alg_{\Lop_{/[m]}[\phi]}(\uV) \arrow{d}{q_{\phi}^{*}} \\
       & \Alg_{\wLop_{\phi}}(\uV).
    \end{tikzcd}
  \]
  In other words, we must show that the canonical map
  $u_{\phi,!} \to q_{\phi}^{*}j_{\phi,!}$ is an equivalence. Using the formula for operadic left Kan extensions, we see that to prove this it suffices to check that for every $X \in \wL_{\phi}$, the functor
  \[ (\bbL_{/[m]})^{\act}_{X/} \to (\bbL_{/[m]})^{\act}_{q(X)/}\]
  is an equivalence of categories. Here the object $X$ is a cellular map that is either of the form $\alpha \colon [a] \to [m]$ or of the form $\beta \colon [b] \to [n]$. In the first case $q(X)$ is again $\alpha$, and both categories are just $(\bbL_{/[m]})^{\act}_{\alpha/}$. In the second case an object of $(\bbL_{/[m]})^{\act}_{X/}$ is a diagram
  \[
    \begin{tikzcd}
      {[b]} \arrow{r} \arrow{d}{\beta} & {[a]} \arrow{d}{\alpha} \\
      {[n]} \arrow{r}{\phi} & {[m]}
    \end{tikzcd}
  \]
  where the top horizontal map is active and $\alpha$ is cellular, while an object of $(\bbL_{/[m]})^{\act}_{q(X)/}$ is a triangle
    \[
    \begin{tikzcd}
      {[b]} \arrow{rr} \arrow{dr}{\phi\beta} & & {[a]} \arrow{dl}{\alpha} \\
       &  {[m]}
    \end{tikzcd}
  \]
  where $\alpha$ is cellular and the horizontal map is active.
\end{proof}

\begin{cor}
  If $\uc{V}$ is a monoidal \icat{} compatible with simplicial colimits, then the functor $\Mor(\uc{V}) \to \MorL(\uc{V})$ is an equivalence.
\end{cor}
\begin{proof}
  We have shown that this functor is a morphism of cocartesian fibrations over $\Dop$, so it suffices to show that it is an equivalence on each fibre. But over $[n] \in \Dop$ we have the restriction $\Algc_{\Dopn}(\uV) \isoto \Alg_{\Lopn}(\uV)$.
\end{proof}

\begin{remark}
  An advantage of Lurie's definition over that of \cite{nmorita} is that it is easy to see that any lax monoidal functor $\uV \to \uW$ induces a lax functor of double \icats{} $\MorL(\uV) \to \MorL(\uW)$ (\ie{} a functor over $\Dop$ that preserves inert cocartesian morphisms). For $\Mor(\uV)$ it is not clear how to obtain any functoriality beyond the obvious one for monoidal functors that preserve simplicial colimits.
\end{remark}

\section{Delooping}\label{sec:deloop}

Let $\uV$ be a monoidal \icat{} compatible with simplicial colimits. In this section we first recall that for any associative algebra $A$ in $\uV$ we can extract from $\Mor(\uV)$ a monoidal \icat{} $\Mor(\uV)_{A}^{\otimes}$, which is a monoidal structure on $A$-$A$-bimodules given by the relative tensor product $\otimes_{A}$. We then prove that when $A$ is the unit $\bbone$, we (unsurprisingly) have a natural equivalence
\[ \Omega_{\bbone}\Mor(\uV) := \Mor(\uV)_{\bbone}^{\otimes} \simeq \uV^{\otimes}\]
of monoidal \icats{}. This shows that $\Mor(\uV)$, or its underlying $(\infty,2)$-category, is a ``delooping'' of $\uV$.


\begin{construction}
  For any \icat{} $\uC$, let $\Dop_{\uC} \to \Dop$ denote the
  cocartesian fibration for the functor $\Dop \to \CatI$ obtained by
  right Kan extending $\uC$ along the inclusion $\{[0]\} \hookrightarrow \Dop$.
  Then $\Dop_{\uC}$ is a double \icat{} (with
  $(\Dop_{\uC})_{[n]} \simeq \uC^{\times n+1}$), and any double \icat{} $\uM$ has a canonical map $\uM \to \Dop_{\uM_{0}}$. An object $X \in \uM_{0}$ induces a functor $\Dop \to \Dop_{\uM_{0}}$, and we define $\uM_{X}^{\otimes}$ as the pullback
  \[
    \begin{tikzcd}
      \uM_{X}^{\otimes} \arrow{d} \arrow{r} & \uM \arrow{d} \\
      \Dop \arrow{r} & \Dop_{\uM_{0}}.
    \end{tikzcd}
  \]
  Then $\uM_{X}^{\otimes}$ is a monoidal \icat{} (since this is a pullback square of cocartesian fibrations over $\Dop$). This gives the monoidal structure on horizontal endomorphism of $X$ in $\uM_{1}$ given by horizontal composition. (See \cite{spanalg}*{Remark 2.17} for a more detailed discussion of this construction.)
\end{construction}

\begin{construction}
  The identity section $\simp \to \Ar(\simp)$ induces for any \icat{} $\uc{C}$ over $\Dop$ a functor
  \[ \oMor'(\uc{C}) \to \uc{C}\] over $\Dop$. On fibres this takes a
  functor $A \colon \Dopn \to \uc{C}$ over $\Dop$ to its value at
  $\id_{[n]} \in \Dopn$.
\end{construction}

\begin{propn}
  If $\uc{V}$ is a monoidal \icat{} compatible with simplicial colimits, then the restriction of this functor to
  \[ \Mor(\uc{V}) \to \uc{V}^{\otimes} \]
  preserves inert morphisms.
\end{propn} 
\begin{proof}
 Given an inert morphism $\phi \colon [n] \to [m]$ in $\simp$ and an algebra $A \colon \Dopn \to \uV^{\otimes}$, this amounts to the induced map
  \[ A(\id_{[m]}) \to (\phi^{*}A)(\id_{[n]}) \simeq A(\phi)\]
  being cocartesian in $\uV^{\otimes}$, which is true since $A$ by assumption preserves inert maps.
\end{proof}

\begin{observation}
  This functor does not preserve cocartesian morphisms in general: Given a composite $\Dop_{/[2]}$-algebra $A$, the image of the cocartesian morphism over $d_{1} \colon [1] \to [2]$ is the canonical map
  \[ A(01) \otimes A(12) \to A(01) \otimes_{A(11)} A(12),\]
  which is typically not an equivalence.
\end{observation}

\begin{cor}
  Let $\uV$ be a monoidal \icat{} compatible with simplicial colimits. For any $A \in \Alg(\uV)$, there is a lax monoidal functor
  \[ \Mor(\uV)_{A}^{\otimes} \to \uV^{\otimes}\]
  given by the canonical maps $M \otimes N \to M \otimes_{A} N$ for $A$-$A$-bimodules $M,N$.
\end{cor}

\begin{propn}\label{propn:deloop}
  Let $\bbone$ be the unit of $\uV$, equipped with its unique associative algebra structure. Then the lax monoidal functor
  \[  \Mor(\uc{V})^{\otimes}_{\bbone} \to \uc{V}^{\otimes}\]
  is an equivalence of monoidal \icats{}.
\end{propn}
\begin{proof}
  We first observe that the functor is monoidal: if $M,N$ are $A$-$A$-bimodules in $\uV$ for any associative algebra $A$, then the relative tensor product $M \otimes_{A} N$ is the colimit of a simplicial diagram in $\uV$ (the ``bar construction''), given at $[n]$ by the iterated tensor product $M \otimes A^{\otimes n} \otimes N$. If $A$ is $\bbone$ then this diagram is constant, so since $\Dop$ is weakly contractible the canonical map
from $M \otimes N$ to the colimit $M \otimes_{\bbone} N$ is an equivalence. To see that this is an equivalence of monoidal \icats{} it then suffices to show that its underlying functor $\Bimod_{\bbone,\bbone}(\uV) := (\Mor(\uc{V})^{\otimes}_{\bbone})_{[1]} \to \uV$ is an equivalence, which is \cite{nmorita}*{Corollary 4.50}.
\end{proof}

\section{Iteration}\label{sec:iterate}

In this section we check that the construction of the double \icat{}
$\Mor(\uV)$ can be iterated ``fibrewise'', giving for an
$E_{n}$-monoidal \icat{} $\uV$ compatible with simplicial colimits an
$(n+1)$-fold \icat{} $\Mor^{n}(\uV)$. We then show by iterating the
``delooping'' equivalence from the previous section that the
$E_{k}$-monoidal structure on $\Mor^{n-k}(\uV)$ obtained by naturality
in $\uV$ is equivalent to that obtained by taking the $k$-fold
endomorphisms $\Omega^{k}_{\bbone}\Mor^{n}(\uV)$ of $\bbone$ in
$\Mor^{n}(\uV)$.

To iterate, we need to know that if $\uV$ is an $E_{2}$-monoidal
\icat{} compatible with simplicial colimits, then $\Mor(\uV)$ is
fibrewise given by monoidal \icats{} compatible with simplicial
colimits. We start by observing that for this purpose it suffices to consider the cartesian product of \icats{} with simplicial colimits:

\begin{notation}
  Let $\LCatIsc$ denote the subcategory of $\LCatI$ whose objects are (large) \icats{} with simplicial colimits and whose objects are the functors that preserve these.
\end{notation}

\begin{observation}[\cite{HA}*{Remark 4.8.1.5}]\label{obs:simpcolimprod}
  Since $\simp^{\op}$ is a sifted \icat{} \cite{HTT}*{Lemma 5.5.8.4},
  a functor $\uC_{1} \times \cdots \times \uC_{n} \to \uD$ preserves
  simplicial colimits in each variable \IFF{} it preserves simplicial
  colimits in the cartesian product --- the non-trivial direction is
  \cite{HTT}*{Proposition 5.5.8.6} (and the converse holds since
  $\Dop$ is weakly contractible).

  Moreover, the subcategory $\LCatIsc$
  of $\LCatI$ inherits cartesian products (and indeed all small limits) from those in $\LCatI$ since a
  functor $\uD \to \uC_{1} \times \cdots \times \uC_{n}$ preserves
  simplicial colimits \IFF{} each of the composites $\uD \to \uC_{i}$
  preserves them.  It follows that an $\uc{O}$-algebra or
  $\uc{O}$-monoid in $\LCatIsct$ can be identified with an
  $\uc{O}$-monoidal \icat{} compatible with simplicial colimits.
\end{observation}

Our construction of $\Mor$ in \cref{sec:double}
above is clearly functorial in monoidal functors that preserve
simplicial colimits, and so gives a functor
\[ \Mor \colon \Mon_{\Dop}(\LCatIsc) \to \LCat_{\infty/\Dop}^{\txt{cocart}} \simeq \Fun(\Dop, \LCatI),\]
(or $\Mon_{\Dop}(\LCatIsc) \to \Seg_{\Dop}(\LCatI)$ since the Segal condition holds),
where the source is the \icat{} of associative monoids in $\LCatIsc$, \ie{} monoidal \icat{} compatible with simplicial colimits.

\begin{propn}\label{propn:algsimpcolim}\
  Let $\uO$ and $\uP$ be \gnsiopds{}, and let $\uU, \uV, \uW$ be monoidal \icats{} compatible with simplicial colimits.
  \begin{enumerate}[(i)]
  \item The \icat{} $\Alg_{\uO}(\uV)$ has simplicial
    colimits, and the forgetful functor to $\Fun(\uO_{[1]}, \uV)$
    detects these.
  \item For any morphism $f \colon \uO \to \uP$, the functor
    \[ f^{*} \colon \Alg_{\uP}(\uV) \to \Alg_{\uO}(\uV)\]
    preserves simplicial colimits.
  \item For any monoidal functor $\phi \colon \uV \to \uW$ that preserves simplicial colimits, the functor
    \[ \phi_{*} \colon \Alg_{\uO}(\uV) \to \Alg_{\uO}(\uW)\]
    preserves simplicial colimits.
  \end{enumerate}
\end{propn}
\begin{proof}
  Part (i) is \cite{enr}*{Corollary A.5.4}. The other parts are then immediate from the fact that simplicial colimits in $\Alg_{\uO}(\uV)$ are detected in $\Fun(\uO_{[1]}, \uV)$.
\end{proof}

\begin{cor}
  The functor $\Mor$ takes $\Mon(\LCatIsc)$ into
  $\Seg_{\Dop}(\LCatIsc) \subseteq \Fun(\Dop, \LCatIsc)$ and preserves
  products.
\end{cor}
\begin{proof}
  Suppose $\uV$ is a monoidal \icat{} compatible with simplicial
  colimits. Then $\Mor(\uV) \colon \Dop \to \LCatI$ factors through
  $\LCatIsc$ by \cref{propn:algsimpcolim}(i) and (ii) and the observation that
  since $\Algc_{\Dopn}(\uV) \hookrightarrow \Alg_{\Dopn}(\uV)$ is a
  colocalization, this full subcategory is closed under simplicial
  colimits. Similarly, \cref{propn:algsimpcolim}(iii) implies that if $\phi \colon \uV \to \uW$ is a monoidal functor that preserves simplicial colimits, then $\Mor(\phi) \colon \Mor(\uV) \to \Mor(\uW)$ is a morphism in $\Fun(\Dop, \LCatIsc)$. To see that $\Mor$ also preserves products, we just note that $\Alg_{\uO}(\blank)$ preserves products as a functor to $\LCatI$ for any $\uO$, and products in $\Fun(\Dop,\LCatIsc)$ are computed pointwise in $\LCatI$ by \cref{obs:simpcolimprod}, 
\end{proof}

\begin{cor}\label{cor:Morsymmon}
  The functor $\Mor \colon \Mon(\LCatIsc) \to \Fun(\Dop, \LCatIsc)$ is symmetric monoidal with respect to the cartesian product on both sides, and so induces for any \gnsiopd{}\footnote{Or indeed any cartesian pattern.} $\uO$ a functor
  \[ \Mon_{\uO \times \Dop}(\LCatIsc) \simeq \Mon_{\uO}(\Mon_{\Dop}(\LCatIsc)) \to \Mon_{\uO}(\Fun(\Dop, \LCatIsc)) \simeq \Fun(\Dop, \Mon_{\uO}(\LCatIsc)).\]
  In particular, if $\uV$ is an  $\uO \times \Dop$-monoidal \icat{} compatible with simplicial colimits, then $\Mor(\uV)$ is a simplicial object in $\uc{O}$-monoidal \icats{} compatible with simplicial colimits. \qed
\end{cor}

We can in particular apply this to $\Dnop$-monoidal \icats{} (where $\Dnop := (\Dop)^{\times n}$), where it allows us allows us to iterate the double Morita \icat{}: we inductively define $\Mor^{n}(\uV)$ by applying $\Mor$ fibrewise to the $(n-1)$-simplicial object  $\Mor^{n-1}(\uV)$ in  monoidal \icats{}.

\begin{cor}
  Iterating the functor $\Mor$ gives a symmetric monoidal functor
  \[ \Mor^{n} \colon \Mon_{\simp^{n,\op}}(\LCatIsc) \to \Seg_{\simp^{n,\op}}(\LCatIsc)\]
  from $E_{n}$-monoidal \icats{} compatible with simplicial colimits to $n$-uple Segal objects in $\LCatIsc$. \qed
\end{cor}

\begin{remark}
  We can also use the functoriality from \cref{cor:Morsymmon} to construct an $(n+m)$-fold \icat{} by applying $\Mor^{n}$ to an $E_{n}$-monoidal $m$-fold \icat{}, as studied in \cite{JohnsonFreydScheimbauerLax}.
\end{remark}

Our remaining goal is to check that the $n$-uple Segal \icat{} $\Mor^{n}(\uV)$ is a $k$-fold delooping of $\Mor^{n-k}(\uV)$.

\begin{remark}
  There are several ways to interpret $\Omega^{k}_{\bbone}\Mor^{n}(\uV)$, but since $\Mor^{n}(\uV)$ is symmetric in the $n$ simplicial coordinates these are all equivalent.
\end{remark}

\begin{propn}
  For $\uV$ an $E_{n}$-monoidal \icat{} compatible with simplicial coliimits, there is a natural equivalence
  \[ \Omega^{k}_{\bbone}\Mor^{n}(\uV) \simeq \Mor^{n-k}(\uV)\] of $E_{k}$-monoidal $(n-k)$-uple \icats{}. 
\end{propn}
\begin{proof}
  We prove this by induction on $k$, interpreting $\Omega^{k}_{\bbone}\Mor^{n}(\uV)$ as applying $\Omega_{\bbone}$ pointwise to $\Mor(\Mor^{n-1}(\uV))$. By \cref{propn:deloop}, this produces the $(n-1)$-fold simplicial object $\Mor^{n-1}(\uV)$ in monoidal \icats{}. We can now use the naturality of the equivalence \cref{propn:deloop} to conclude that $\Omega^{k}_{\bbone}\Mor^{n}(\uV)$ is equivalent to $\Omega^{k-1}_{\bbone}\Mor^{n-1}(\uV) \simeq \Mor^{n-k}(\uV)$ as $E_{k}$-monoidal \icats{}.
\end{proof}

\end{document}